\documentclass[10pt]{article}
\usepackage{color}
\usepackage{amssymb}
\usepackage{amsthm,array,amsfonts}
\usepackage{amsmath}
\usepackage{verbatim}
\usepackage{tikz-cd}
\usepackage{enumitem}
\setitemize{itemsep=0pt}
\setenumerate{itemsep=0pt}

\usepackage[justification=centering]{caption}
\usepackage{hyperref}

\usepackage{stmaryrd}
\hypersetup{
    colorlinks,
    citecolor=blue,
    filecolor=black,
    linkcolor=black,
}

\usepackage{accents}

\newtheorem{theo}{Theorem}[section]
\newtheorem{thm}{Theorem}[section]
\newtheorem{prop}[theo]{Proposition}

\newtheorem{lemma}[theo]{Lemma}

\newtheorem{cor}[theo]{Corollary}
\newtheorem{question}[theo]{Question}

\newtheorem{conjecture}{Conjecture}  
\newtheorem{conj}[conjecture]{Conjecture}

\theoremstyle{remark}

\newtheorem{rmk}[theo]{Remark}
\newtheorem{example}[theo]{Example}

\newcommand{\BC}{{\mathbb{C}}}

\newcommand{\BQ}{{\mathbb{Q}}}
\newcommand{\BR}{{\mathbb{R}}}

\newcommand{\BZ}{{\mathbb{Z}}}

\newcommand{\CE}{{\mathcal E}}
\newcommand{\CF}{{\mathcal F}}

\newcommand{\CL}{{\mathcal L}}

\newcommand{\CO}{{\mathcal O}}

\newcommand{\Fg}{{\mathfrak{g}}}

\newcommand{\Fq}{{\mathfrak{q}}}

\newcommand{\ch}{\mathsf{ch}}

\newcommand{\id}{\mathrm{id}}

\newcommand{\pt}{{\mathsf{p}}}
\newcommand{\td}{{\mathrm{td}}}

\newcommand{\Mbar}{{\overline M}}
\newcommand\ev{\operatorname{ev}}
\newcommand{\Pic}{\mathop{\rm Pic}\nolimits}

\DeclareFontFamily{OT1}{rsfs}{}
\DeclareFontShape{OT1}{rsfs}{n}{it}{<-> rsfs10}{}
\DeclareMathAlphabet{\curly}{OT1}{rsfs}{n}{it}

\newcommand\Hom{\operatorname{Hom}}
\newcommand{\p}{\mathbb{P}}

\newcommand\Hilb{\operatorname{Hilb}}

\newcommand\Coh{\operatorname{Coh}}

\newcommand{\vir}{\mathsf{vir}}

\DeclareFontFamily{OT1}{rsfs}{}
\DeclareFontShape{OT1}{rsfs}{n}{it}{<-> rsfs10}{}
\DeclareMathAlphabet{\curly}{OT1}{rsfs}{n}{it}

\newcommand\End{\operatorname{End}}

\newcommand{\1}{v_{\varnothing}}

\newcommand{\Aut}{\mathrm{Aut}}

\newcommand{\Mon}{\mathsf{Mon}}
\newcommand{\DMon}{\mathsf{DMon}}

\newcommand{\SO}{\mathrm{SO}}

\newcommand{\ST}{\mathsf{ST}}

\newcommand{\vacuum}{\1}

\newcommand{\tr}{\operatorname{tr}}
\newcommand{\Tr}{\operatorname{Tr}}

\newcommand{\pr}{\operatorname{pr}}

\newcommand{\FM}{\operatorname{FM}}

\title{Lagrangian planes in hyperk\"ahler varieties of $K3^{[n]}$-type}
\author{Georg Oberdieck}
\date{\today}

\begin{document}
\maketitle
\setcounter{section}{0}

\begin{abstract}
Jieao Song recently conjectured a formula for the class of a Lagrangian plane on a hyperk\"ahler variety of $K3^{[n]}$-type
in terms of the class of a line on it.
We give a proof of this conjecture if the line class is primitive.
\end{abstract}

\section{Introduction}
An (irreducible) hyperk\"ahler variety $X$ is a simply-connected smooth projective variety such that $H^0(X,\Omega_X^2)$ is generated by a holomorphic-symplectic form.
A Lagrangian plane $P \subset X$ is a Lagrangian submanifold with $P \cong \p^n$, where $\dim(X)=2n$.
Let $\ell \in H_2(X,\BZ)$ be the class of a line in $P$.
In the context of characterizing ample classes in terms of the intersection properties of the Beauville-Bogomolov-Fujiki (BBF)
pairing on $H^2(X,\BZ)$ and its $\BQ$-linear extension to $H_2(X,\BZ) \cong H^2(X,\BZ)^{\ast}$,
Hassett and Tschinkel proposed that the norm $(\ell, \ell)$ should be a universal constant
depending only on the deformation type of $X$ \cite{HT2}.
For hyperk\"ahler varieties deformation equivalent to the Hilbert scheme of $n$ points on a K3 surface (we say they are of $K3^{[n]}$-type)
the constant was shown to be $(\ell, \ell) = -(n+3)/2$ by \cite{HT, HHT, BJ} for $n=2,3,4$ and eventually for all $n$ by Bakker \cite{Bakker}.
For generalized Kummer fourfolds we have $(\ell, \ell) = -3/2$ by \cite{HT3}.
For arbitrary hyperk\"ahler varieties, Song \cite[Conj. 2.2.12]{SongPhD} conjectured that 
\[ (\ell, \ell) = -2 r_X, \quad r_X = \frac{(2n-1) C(c_2(T_X))}{24 C(1)} \]
where $C(\alpha)$ is the Fujiki constant of a monodromy invariant class $\alpha$.

In this paper we study what can be said about the class $[P] \in H^{2n}(X,\BZ)$.
Let
\[ L = ( \ell, - ) \in H_2(X,\BQ)^{\vee} \cong H^2(X,\BQ) \]
be the class dual to $\ell$ with respect to the BBF form.
The class $L \in H^2(X,\BQ)$ is characterized by $(L, x) = \int \ell \cup x$ for all $x \in H^2(X,\BZ)$.
For $K3^{[n]}$-type an explicit formula for $[P]$ was found in terms of powers of $L$ and absolute Hodge classes
in case $n=2,3,4$ by Hassett-Tschinkel \cite{HT}, Harvey-Hassett-Tschinkel \cite{HHT}, Bakker-Jorza \cite{BJ}.
Recently, Song conjectured an extension of their formulas to any $n$ in the following simple and beautiful way.
Let $[ \gamma ]_k \in H^{2k}(X)$ denote the complex degree $k$ (cohomological degree $2k$) component of a class $\gamma \in H^{\ast}(X)$.

\begin{conj}[{Song, \cite[Conj. 2.2.10]{SongPhD}}] \label{conj:Song} Let $X$ be a hyperk\"ahler variety of $K3^{[n]}$-type and let $P \subset X$ be a Lagrangian plane. Then
\begin{equation} [ P ] = \left[ \exp(L) \sqrt{\td_X} \right]_{n}, \label{90sifsdf} \end{equation}
where $L \in H^2(X,\BQ)$ is the class dual to the class $\ell$ of a line on $P$.
\end{conj}

The main result of this paper is a proof of his conjecture in the primitive case.

\begin{thm} \label{thm:main} Conjecture~\ref{conj:Song} holds if $\ell$ is primitive (i.e. indivisible in $H_2(X,\BZ)$).
\end{thm}

By Bakker \cite[Remark 28]{Bakker} there exists cases where the line class $\ell$ is imprimitive.
The expectation that the formula for $[P]$ is independent of the divisibility of $\ell$ is remarkable.
It suggests connections to the independence of Gromov-Witten invariants
of $K3^{[n]}$-type hyperk\"ahler varieties from the divisibility proven in \cite{QuasiK3}. We discuss this further in Section~\ref{sec:open questions}.
There we also comment on Song's formula for the case of arbitrary hyperk\"ahler varieties
and potential lifts to the Chow ring.
The remaining sections of the paper deal with preliminaries (Section~\ref{sec:preliminaries}) and the proof (Section~\ref{sec:proof}).
The appendix contains a discussion of Mukai vectors for orbifolds.

To hint at the proof of Theorem~\ref{thm:main},
let us write $\CO(L)$ for the (fractional) line bundle with first Chern class $L$.
Let $v(E) = \ch(E) \sqrt{\td}_X$ be the Mukai vector of an object $E \in D^b(X)$.
Then \eqref{90sifsdf} can be rewritten as
\[ [P] = \left[ v(\CO(L)) \right]_{n}. \]
This suggests that auto-equivalences should play a key role in the proof, and this is indeed the case.
More precisely, our proof here is heavily inspired by ideas of Beckmann in \cite[Sec.7]{Beckmann1}.
We can obtain stronger results as in \cite{Beckmann1} (where only the projection of $[P]$ to the Verbitsky component was computed)
by combining a recent result of Markman \cite{Markman2} with a result of the author \cite{OLLV}
to describe how autoequivalences of $S^{[n]}$ induced by autoequivalences of $S$ using Ploog's construction \cite{Ploog} act on cohomology.
Otherwise the strategy is parallel, and in particular Section~\ref{sec:key computation} is a detailed exposition of a computation in \cite[Prop.7.2]{Beckmann1}.

\subsection{Acknowledgements}
I thank Jieao Song for sharing his conjecture early on and Thorsten Beckmann and Eyal Markman for discussions on hyperk\"ahler varieties.
The idea that Theorem~\ref{thm:main} can be proven along the lines presented here was also observed independently by Beckmann.
The author is funded by the Deutsche Forschungsgemeinschaft (DFG) - OB 512/1-1.

\section{Preliminaries}
\label{sec:preliminaries}
\subsection{The LLV algebra and Mukai lattice}
Let $X$ be a (irreducible) hyperk\"ahler variety.
For $\lambda \in H^2(X,\BZ)$ we let
$e_{\lambda} \in \End H^{\ast}(X)$ be the operator of cup product with $\lambda$.
Let $h \in \End H^{\ast}(X)$ be the grading operator which acts on $H^i(X)$ by multiplication by $i - \dim(X)$.
The Looijenga-Lunts-Verbitsky (LLV) Lie algebra $\Fg(X)$ is the Lie subalgebra of $\End H^{\ast}(X)$
generated by all Lefschetz triples $(e_{\lambda}, h, f_{\lambda})$ for all $\lambda \in H^2(X)$ with $(\lambda, \lambda)>0$,
where $f_{\lambda} \in \End H^{\ast}(X)$ is the unique operator such that $[e_{\lambda}, f_{\lambda}]=h$ and $[h,f_{\lambda}]=-2 f_{\lambda}$,
see \cite{LL,V}.

The LLV lattice of $X$ is the vector space
\[ \widetilde{H}(X,\BQ) = \BQ \alpha \oplus H^2(X,\BQ) \oplus \BQ \beta \]
equipped with the intersection pairing which restricts to the BBF form on $H^2(X,\BQ)$, lets $\alpha,\beta$ be orthogonal to $H^2(X,\BQ)$
and satisfies $(\alpha,\alpha) = (\beta,\beta) =0 $ and $(\alpha, \beta)=-1$.
For $X$ a K3 surface, we have the natural identification $\widetilde{H}(X) = H^{\ast}(S,\BQ)$ where $\alpha = 1$ is the unit, and $\beta = \pt \in H^4(S,\BZ)$ is the class of a point.
In higher dimension the geometric meaning of $\widetilde{H}(X,\BQ)$ is more subtle, see \cite{Talman, Beckmann1, Markman1} for more details.

Let $\DMon(X) \subset O(H^{\ast}(X,\BZ))$ be the 
subgroup generated by all morphism $\gamma_2^{-1} \circ F^H \circ \gamma_1$ where
$\gamma_i : H^{\ast}(X) \to H^{\ast}(X_i)$ are parallel transport operators,
and $F : D^b(X_1) \to D^b(X_2)$ is an auto-equivalence.
By work of Talman \cite{Talman} there is a natural functor
\[ \widetilde{H} : \DMon(X) \to O(\widetilde{H}(X,\BQ)). \]
which controls the action of $\DMon(X)$ on the subring of $H^{\ast}(X,\BQ)$ generated by $H^2(X,\BQ)$.

\subsection{Hilbert schemes}
Let $S$ be a K3 surface and let $S^{[n]}$ be the Hilbert scheme of $n$ points on $S$.
There is a natural isomorphism
\[ H^2(S,\BZ) \oplus \BZ \delta \xrightarrow{\cong} H^2(S^{[n]},\BZ), \quad (\lambda,k \delta) \mapsto \theta(\lambda) + k \delta \]
where $2 \delta$ is the class of the locus of non-reduced subschemes, and
\[ \theta : H^2(S,\BZ) \cong H^2(S^{(n)},\BZ) \to H^2(S^{[n]},\BZ) \]
is the canonical morphism, where the second map is given by pullback along the Hilbert-Chow morphism $S^{[n]} \to S^{(n)}$.
%
If $\CL \in \Pic(S)$ is a line bundle, we let $\CL_n \in \Pic(S^{[n]})$ be the line bundle
with $c_1(\CL_n) = \theta(c_1(\CL))$.

Let also $\Xi_n \subset S^{[n]} \times S$ be the universal family of the Hilbert scheme,
and let $p,q$ be the projection of $S^{[n]} \times S$ to the factors.
Given a vector bundle $V$ on $S$, we let
\[ V^{[n]} := p_{\ast}( q^{\ast}(V) \otimes \CO_{\Xi_n} ) \]
be the tautological bundle associated to $V$.
If $\CL$ is a line bundle, then $\det( \CL^{[n]}) = \CL_n \otimes \CO(-\delta)$.

\subsection{The BKR equivalence} \label{section:BKR}
By Haiman's work \cite{Haiman} the Hilbert scheme $S^{[n]}$
is isomorphic to the Nakamura $G$-Hilbert scheme $\Hilb^G(S^n)$ for the permutation action of $G:=S_n$ on $S^n$.
Let $Z_n \subset S^{[n]} \times S^n$ be the universal family of $\Hilb^G(S^n)$.
By construction, $Z_n$ is is $G$-equivariant, and finite and flat of degree $n!$ over $S^{[n]}$.
The Bridgeland-King-Reid equivalence \cite{BKR} which we use here in the convention of Krug \cite{Krug} is:
\[ \Psi : D_G(S^n) \xrightarrow{\cong} D^b(S^{[n]}), \quad \CF \mapsto \pr_{2\ast}( \pr_1^{\ast}(\CF) \otimes \CO_Z )^G \]
where $D_G(S^n)$ is the equivariant category of $D^b(S^n)$ with respect to the $G$-action.

Ploog and Sosna \cite{Ploog, PS} use $\Psi$ to construct an injective group homomorphism
\[ \Aut(D^b(S)) \times \BZ_2 \to \Aut( D^b(S^{[n]})), \quad (F,1) \mapsto F^{[n]}, \quad (\id, -1) \mapsto F_{\chi}. \]
We recall the construction.
Given an object $E \in D^b(Y)$ on a smooth projective variety $Y$,
the box product
\[ E^{\boxtimes n} = \pi_1^{\ast}(E) \otimes \cdots \otimes \pi_n^{\ast}(E) \in D^b(Y^n), \]
carries a natural $S_n$-linearization by assigning to $g \in S_n$ the isomorphism 
\begin{equation} \label{linearization}
\begin{aligned}
 g^{\ast}(E^{\boxtimes n}) & \cong \pi_{g(1)}^{\ast}(E) \otimes \cdots \pi_{g(n)}^{\ast}(E) \\
& \cong \pi_1^{\ast}(E) \otimes \cdots \otimes \pi_n^{\ast}(E) \\
& \cong E^{\boxtimes n},
\end{aligned}
\end{equation}
where $\pi_i$ are the projection of $Y^n$ to the $i$-th factor.
We let $(E^{\boxtimes n},1)$ denote this representation.
Let $\chi$ be the non-trivial character of $S_n$. We let 
$(E^{\boxtimes n},-1) = (E^{\boxtimes n},1) \otimes \chi$ be the $\chi$-twisted linearization.

Given a Fourier-Mukai transform $F = \FM_{\CE} : D^b(S) \to D^b(S)$ with Fourier-Mukai kernel $\CE \in D^b(S \times S)$,
we have an induced Fourier-Mukai kernel $(\CE^{\boxtimes n},1) \in D^b( (S \times S)^n )$,
and hence an induced autoequivalence
\begin{equation} F^{\boxtimes n} := \FM_{(\CE^{\boxtimes n},1)} : D_G(S^n) \to D_G(S^n), \quad A \mapsto \pr_{2 \ast}( \pr_1^{\ast}(A) \otimes (\CE^{\boxtimes n},1)) \label{induced box transform} \end{equation}
where the pullback, tensor product, and pushforward are taken equivariantly.
The first part of the Ploog-Sisna map is now defined by:
\begin{gather*}
F^{[n]} = \FM_{\CE}^{[n]} := \Psi \circ \FM_{\CE^{\boxtimes n}} \circ \Psi^{-1}.
\end{gather*}
Further, tensoring with $\chi$ 
commutes with $F^{\boxtimes n}$. We set
\[ F_{\chi} := \Psi \circ (\chi \otimes ( - )) \circ \Psi^{-1}. \]

\begin{rmk}
We adopt in \eqref{linearization} the natural sign convention for graded tensor products:
If $V,W$ are graded vector spaces and $a \in V$ and $b \in W$ are homogeneous, then under the isomorphism
$V \otimes W \cong W \otimes V$ we send $a \otimes b$ to $(-1)^{|a| |b|} b \otimes a$, where $|a|$ is the degree of $a$.

For example, $\CO_S[-1]$ is of odd degree, so we get
\begin{equation} (\CO_S[-1]^{\boxtimes n}, 1) \cong (\CO_S^{\boxtimes n},-1) [ -n ]. \label{odd equiv} \end{equation}
\end{rmk}

\begin{rmk}
For any $B \in D^b(S)$ and $F \in \Aut(D^b(S))$ we have by construction
\begin{equation} F^{\boxtimes n}\big( ( B^{\boxtimes n}, 1) \big) = (F(B)^{\boxtimes n}, 1). \label{compat} \end{equation}

For example, consider the spherical twist of $S$ along the structure sheaf $\CO_S$,
\[ \ST_{\CO_S} : D^b(S) \to D^b(S), \quad E \mapsto \mathrm{Cone}( R \Hom(\CO_S, E) \otimes \CO_S \to E ). \]
Then $\ST_{\CO_S}(\CO_S) = \CO_S[-1]$ hence using \eqref{odd equiv} we get that
\[ \ST_{\CO_S}^{\boxtimes n}( \CO_S^{\boxtimes n}, 1)= (\ST_{\CO_S}(\CO_S)^{\boxtimes n}, 1) = (\CO_S[-1]^{\boxtimes n}, 1) = (\CO_S^{\boxtimes n}, -1)[-n]. \]
\end{rmk}

\begin{rmk} \label{rmk:commute with tensor}
For a line bundle $\CL \in \Pic(S)$ let $F_{\CL} = \CL \otimes (- )$ be
the autoequivalence acting by tensor product with $\CL$.
Then as explained in \cite[Remark 3.11]{Krug} we have
\[ F_{\CL}^{[n]} = \CL_n \otimes ( - ). \]
\end{rmk}

\subsection{The BKR equivalence in cohomology} \label{subsec:BKR in cohomology}
We also have a cohomological version of Ploog's construction.
Let $H^{\ast}_G(S^n)$ be the orbifold cohomology of $S^n$,
for which we refer to Appendix~\ref{sec:orbifold mukai}.
Then as discussed in Section~\ref{sec:appendix BKR} $\Psi$ induces an isomorphism on cohomology
\[ \Psi^H : H^{\ast}_G(S^n) \xrightarrow{\cong} H^{\ast}(S^{[n]}). \]
We define
\[ \End(H^{\ast}(S)) \to \End(H^{\ast}(S^{[n]})), \quad \phi \mapsto \phi^{[n]} := \Psi^H \circ \phi^{\boxtimes n} \circ (\Psi^{H})^{-1} \]
where $\phi^{\boxtimes n} : H_G(S^n) \to H_G(S^n)$ is the endomorphism induced by $\phi$, see Section~\ref{subsec:symmetric action on Sn}.

The following follows from the appendix:
\begin{lemma} \label{lemma:induced}
For any $F \in \Aut D^b(S)$, we have $(F^{[n]})^H = (F^H )^{[n]}$.
\end{lemma}

Recall that $\DMon(S) = O^{+}(H^{\ast}(S,\BZ))$ and consider the natural embedding
\[ \iota : O(H(S,\BQ)) \to O(\widetilde{H}(S^{[n]},\BQ)), \quad g \mapsto \iota(g) \]
where we let $\iota(g)$ act on the image of $\id_{\alpha} \oplus \theta \oplus \id_{\beta} : H^{\ast}(S,\BQ) \subset \widetilde{H}(S^{[n]},\BQ)$ by $g$
and by the identity on the orthogonal complement.

\begin{lemma} \label{lemma:Htilde conjugate}
$\widetilde{H}( \phi^{[n]} ) = \det(\phi)^{n+1} B_{-\delta/2} \circ \iota(\phi) \circ B_{\delta/2}$
\end{lemma}
\begin{proof}
This can be found in \cite[Theorem 7.4]{Beckmann1} or \cite[Theorem 12.2]{Markman1}.
\end{proof}

This leads to the following useful result, proven by Markman recently in \cite{Markman2}.
Let
\[ \rho : \SO(\widetilde{H}(S^{[n]},\BC)) \to \SO(H^{\ast}(S^{[n]},\BC)) \]
be the (integrated) LLV respresentation.
Define the conjugated morphism:
\[ \Theta_n(\phi) = B_{\delta/2} \circ \phi^{[n]} \circ B_{-\delta/2}. \]
\begin{prop} \label{prop:Markman} For any $g \in \DMon(S)$ of determinant $1$ we have:
\[ \Theta_n(g) = \rho( \iota(g)). \]
\end{prop}
\begin{proof}
This is proven in \cite{Markman2} but we sketch the proof here for convenience.
The claimed equality makes sense for all of $\SO(H^{\ast}(S,\BC))$, so we only need to check the derivative at the origin.
Moreover, it suffice it to check it on the generators $e_{\lambda}$ and $f_{\lambda}$
of the tangent space $\mathfrak{so}(H^{\ast}(S,\BC))$.
For $e_{\lambda}$ this is clear by Remark~\ref{rmk:commute with tensor}.
Hence it suffices to compute the derivative of $\Theta_n(\exp(f_{\lambda} t))$ at $t=0$
and check that it equals $f_{\theta(\lambda)}$.

Consider the action on cohomology $\phi := ( \ST_{\CO_S} )^{H}$ of the spherical twist $\ST_{\CO_S}$.
Then $\phi$ acts by the identity on $H^2(S)$ and sends $1, \pt$ to $-\pt, 1$ respectively. 
Hence in $\mathfrak{so}(H^{\ast}(S,\BC))$ we have
\[ \phi^{-1} e_{\lambda} \phi = \frac{- (\lambda, \lambda)}{2} f_{\lambda}. \]
Further by Lemma~\ref{lemma:induced} we know that $\phi^{[n]}$ is the induced action of an auto-equivalence on $D^b(S^{[n]})$.
Thus by the main result of \cite{Talman} $\phi^{[n]}$ and hence $\Theta_n(\phi)$ conjugates elements of the LLV algebra $\Fg(S^{[n]})$.
Together with Lemma~\ref{lemma:Htilde conjugate} one finds by the $\rho$-equivariance of $\widetilde{H}$ that:
\[
\Theta_n(\phi)^{-1} e_{\theta(\lambda)} \Theta_n(\phi) = \frac{-(\lambda,\lambda)}{2} f_{\theta(\lambda)}.
\]
We get that:
\begin{align*}
\frac{d}{dt}\Big|_{t=0} \Theta_n(\exp(f_{\lambda} t))
& = \frac{d}{dt}\Big|_{t=0} \Theta_n\left(\exp \left( \frac{-2t}{(\lambda,\lambda)} \phi^{-1} e_{\lambda} \phi \right) \right) \\
& = \frac{d}{dt}\Big|_{t=0} \Theta_n(\phi^{-1}) \Theta_n\left( \exp \left( \frac{-2t}{(\lambda,\lambda)} e_{\lambda} \right) \right) \Theta_n(\phi) \\
& = \frac{d}{dt}\Big|_{t=0} \Theta_n(\phi^{-1}) \exp\left( \frac{-2t}{(\lambda,\lambda)} e_{\theta(\lambda)} \right) \Theta_n(\phi) \\
& = \Theta_n(\phi^{-1})\left[ \frac{d}{dt}\Big|_{t=0} \exp\left( \frac{-2t}{(\lambda,\lambda)} e_{\theta(\lambda)} \right) \right] \Theta_n(\phi) \\
& = \Theta_n(\phi^{-1})\frac{-2}{(\lambda,\lambda)} e_{\theta(\lambda)} \Theta_n(\phi) \\
& = f_{\theta(\lambda)} \\
& = \rho( \iota(f_{\lambda}) )
\end{align*}

\end{proof}

To obtain the description of $\Theta_n(g)$ for all $g \in \DMon(S)$, also of determinant $-1$,
we can use the following argument:
Let $h \in \Mon(S)$ be any monodromy operator of $S$ with determinant $-1$.
Then for any $g \in \DMon(S)$ of determinant $1$ 
Proposition~\ref{prop:Markman} determines $\Theta_n(gh)$.
Moreover, $\Theta_n(h) = h^{[n]}$ is just the induced parallel transport operator.
Hence:
\begin{equation} \label{abc} \Theta_n(g) = \Theta_n(gh) \Theta_n(h)^{-1} = \rho( \widetilde{\iota}(gh) ) \circ (h^{[n]})^{-1}. \end{equation}

For $k \geq 1$ and $\gamma \in H^{\ast}(S)$ consider the Nakajima creation operators \cite{Nak} (here in the convention of \cite{NOY})
\[ \Fq_k(\gamma) : H^{\ast}(S^{[m]}) \to H^{\ast}(S^{[m+k]}) \]
and let $\vacuum \in H^{\ast}(S^{[0]})$ be the unit ('the vacuum vector'). We find the following description:
\begin{cor}  \label{cor:Markman}
For any $g \in \DMon(S)$ we have
\[ \Theta_n(g) \left( \Fq_{k_1}(\gamma_1) \cdots \Fq_{k_{\ell}}(\gamma_{\ell}) \vacuum \right) = 
\Fq_{k_1}( \psi_{k_1}(g) \gamma_1) \cdots \Fq_{k_{\ell}}( \psi_{k_{\ell}}(g)  \gamma_{\ell}) \vacuum \]
where $\psi_n(g) = n^{h/2} \circ g \circ n^{-h/2}$ and we let $n^{h}(\gamma) = n^{h(\gamma)} \gamma$ for any homogeneous $\gamma$.
\end{cor}
\begin{proof}
Consider first the determinant $1$ case.
This equality makes sense for all elements $g \in \SO(H^{\ast}(S,\BC))$ so 
it suffices to check this for the differential.
Take $g=\exp(e_{\lambda} t)$ or $g=\exp(f_{\lambda} t)$ for some $\lambda \in H^2(S)$ and compute the derivative at $t=0$ on the left hand side.
For $\rho(e_{\lambda})$ an explicit expression in terms of Nakajima operators
was found in \cite{Lehn}, and for $\rho(f_{\lambda})$ in \cite{OLLV}.
Using these references we precisely get the right hand side. This concludes the proof in determinant $1$.
For determinant $-1$ use \eqref{abc} and that the Nakajima operators are equivariant with respect to the monodromy operators of $S$.
\end{proof}

\begin{rmk}
The twist $\psi_n(g)$ can be explained by the computations in the orbifold cohomology of $S^n$,
see Section~\ref{subsec:symmetric action on Sn}.
\end{rmk}

\section{Proof of main result} \label{sec:proof}
\subsection{Reduction}
By work of Bakker \cite{Bakker} there exists a unique monodromy orbit
for Lagrangian planes $P \subset X$ in hyperk\"ahler varieties of $K3^{[n]}$-type, such that the class of the line in $P$ is primitive in $H_2(X,\BZ)$.
Because the equality in Theorem~\ref{thm:main} in one case implies the same equality for all deformations of the pair $(X,P)$,
it suffices to prove the theorem for a single case.
We hence can specialize $X$ to $S^{[n]}$, assume that $S$ contains a $(-2)$-curve $C \cong \p^1$, and
for the Lagrangian plane $P$ take
\[ \p^n \cong C^{[n]} \subset S^{[n]}. \]
The class of a line $\ell$ in $C^{[n]}$ is
\[ \ell = \ell_n - (n-1)A \in H_2(S^{[n]},\BZ) \]
where $\ell_n$ is the class of the locus of subschemes which are incident to $C$ and $n-1$ fixed distinct points away from $C$,
and $A$ is the class of a fiber of the Hilbert-Chow morphism over a generic point in the singular locus.
Using that $\delta \cdot A = -1$ and setting $\CL := \CO_S(C)$ we hence get the dual class
\[ L = \ell^{\vee} = \theta(c_1(\CL)) - \frac{\delta}{2} = c_1(\CL_n \otimes \CO(-\delta/2)) \in H^2(S^{[n]},\BQ). \]
To prove Theorem~\ref{thm:main} it hence suffices to check the following explicit formula:
\begin{equation} [ C^{[n]} ] = \left[ v( \CL_n \otimes \CO(-\delta/2) ) \right]_{n}. \label{reduction} \end{equation}

\subsection{Key computation} \label{sec:key computation}
Consider again the spherical twist along the structure sheaf:
\[ \ST_{\CO_S} : D^b(S) \to D^b(S), \quad E \mapsto \mathrm{Cone}( R \Hom(\CO_S, E) \otimes \CO_S \to E ). \]
The following basic computation is our main input:
\begin{prop} \label{prop:Key}
Let $\CL = \CO_S(C)$ for a $(-2)$-curve $C \subset S$. Then
\[ \ST_{\CO_S}^{[n]}( \CL_n \otimes \CO(-\delta) ) = \iota_{\ast}( \omega_{\p^n} ) \]
where $\iota : \p^n \cong C^{[n]} \to S^{[n]}$ is the inclusion.
\end{prop}

\begin{rmk}
The proof of the proposition below is a detailed exposition of a computation done in \cite[Prop.7.2]{Beckmann1}.
(This exposition corrects a subtle sign issue that can be found in the first two arXiv versions of \cite{Beckmann1}.).
\end{rmk}

For the proof we need some preparation.
For $I \subset \{ 1, \ldots, n \}$ let $\pi_I : S^n \to S^{|I|}$ be the projection to the components indexed by $I$.
For $L \in \Pic(S)$ the pullback
\[ \pi_{1 \cdots k}^{\ast}( L^{\boxtimes k}, -1 ) \cong \pi_{1 \cdots k}^{\ast}( L^{\boxtimes k}, -1 ) \otimes \pi_{(k+1) \cdots n}^{\ast}( \CO_S^{\boxtimes n-k} , 1) \]
carries a natural $S_k \times S_{n-k}$ linearization. Following \cite[Defn.3.4]{Krug} we define
\[ W^k(L) = \mathrm{Ind}_{S_k \times S_{n-k}}^{S_n}( \pi_{1 \cdots k}^{\ast}( L^{\boxtimes k}, -1 ) ) 
=
\Big( \bigoplus_{\substack{I \subset \{ 1, \ldots, n \} \\ |I| = k }} \mathrm{pr}_I^{\ast}(L^{\boxtimes k}), \sigma \Big) \in D_{G}(S^n).
\]
where $\mathrm{Ind}_{S_k \times S_{n-k}}^{S^n}$ is the induction functor of the inclusion $S_k \times S_{n-k} \subset S_n$,
see \cite[Sec.2.2]{Krug}.
Krug's main result (\cite[Thm 1.1]{Krug}) then shows that
\begin{equation} \Psi( W^k(L)) = \wedge^k L^{[n]}. \label{Krug result} \end{equation}

\begin{lemma} \label{lemma:key helper}
For any $k \in \{ 0, \ldots n \}$ we have
\[ W^k(L) \otimes \chi = (L^{\boxtimes n},1) \otimes W^{n-k}(L^{\vee}). \]
\end{lemma}
\begin{proof}
We have
\begin{align*}
W^k(L) \otimes \chi & = \mathrm{Ind}_{S_k \times S_{n-k}}^{S_n}\Big( \pi_{1 \cdots k}^{\ast}( L^{\boxtimes k}, 1 ) \otimes \pi_{(k+1) \cdots n}^{\ast}(\CO_S^{\boxtimes (n-k)},-1)  \Big) \\
& = (L^{\boxtimes n},1) \otimes \mathrm{Ind}_{S_k \times S_{n-k}}^{S_n}\Big( \pi_{1 \cdots k}^{\ast}( \CO_S^{\boxtimes k}, 1 ) \otimes \pi_{(k+1) \cdots n}^{\ast}((L^{\vee})^{\boxtimes (n-k)},-1)  \Big) \\
& = (L^{\boxtimes n},1) \otimes W^{n-k}(L^{\vee}).
\end{align*}
\end{proof}

\begin{proof}[Proof of Proposition~\ref{prop:Key}]
Since we have
\[ \ST_{\CO_S}(\CO_S(C)) = \CO_C(C), \]
we get by \eqref{compat} that
\[ \ST_{\CO_S}^{\boxtimes n}( \CO(C)^{\boxtimes n},1 ) = (\CO_C(C)^{\boxtimes n},1). \]
Since tensoring with $\chi$ commutes with any $F^{\boxtimes n}$ for $F \in \Aut(D^b(S))$, therefore
\[ \ST_{\CO_S}^{\boxtimes n}( \CO(C)^{\boxtimes n},-1 ) = (\CO_C(C)^{\boxtimes n},-1). \]

We now apply $\Psi$ and compute both sides. The claim will follow.

\vspace{7pt}
\emph{Step 1: The left hand side.}  We have
\begin{align*}
\Psi \circ \ST_{\CO_S}^{\boxtimes n}( \CO(C)^{\boxtimes n},-1 )
& = \ST_{\CO_S}^{[n]} \Psi( \CL^{\boxtimes n} \otimes (\CO_S^{\boxtimes n}, -1) ) \\
& = \ST_{\CO_S}^{[n]}( \CL_n \otimes \Psi( \CO_S^{\boxtimes n}, -1) ) \\
& = \ST_{\CO_S}^{[n]}( \CL_n \otimes \Psi( W^n(\CO_S) ) ) \\
& = \ST_{\CO_S}^{[n]}( \CL_n \otimes \det( \CO_S^{[n]}) ) \\
& = \ST_{\CO_S}^{[n]}( \CL_n \otimes \CO_{S^{[n]}}(-\delta) ),
\end{align*}
where we used Remark~\ref{rmk:commute with tensor} in the second equation,
the definition of $W^k$ in the third,
Krug's result \eqref{Krug result} in the forth.

\vspace{7pt}
\emph{Step 2: The right hand side.}  Using the resolution $0 \to \CO_S(-C) \to \CO_S \to \CO_C \to 0$,
and taking the $n$-th box-product, we obtain the $S_n$-equivariant Koszul resolution
\[
\Big[ W^n(\CO_S(-C)) \to  \ldots \to W^1(\CO_S(-C)) \to W^0(\CO_S(-C)) \Big] \cong (\CO_C^{\boxtimes n}, 1).
\]
Tensoring with the non-trivial character $\chi$ and applying Lemma~\ref{lemma:key helper} we find:
\[
(\CO_S(-C)^{\boxtimes n}, 1) \otimes 
\Big[ W^0(\CO_S(C)) \to  \ldots \to W^{n-1}(\CO_S(C)) \to W^n(\CO_S(C)) \Big] \cong (\CO_C^{\boxtimes n}, -1).
\]
Hence we get
\[ \Big[ W^0(\CO_S(C)) \to  \ldots \to W^{n-1}(\CO_S(C)) \to W^n(\CO_S(C)) \Big] \cong (\CO_C(C)^{\boxtimes n}, -1). \]
Now we apply $\Psi$ and use Krug's result \eqref{Krug result}, to get
\begin{equation} \Big[ \CO_{S^{[n]}} \to \CO_S(C)^{[n]} \to  \ldots \to \wedge^{n-1}(\CO_S(C)^{[n]}) \to \det(\CO_S(C)^{[n]}) \Big] \cong \Psi(\CO_C(C)^{\boxtimes n}, -1). \label{abc222} \end{equation}
The term $\CO_{S^{[n]}}$ appears here in degree $-n$.

It is well-known that $\CO_{C^{[n]}}$ is cut out by a regular section of $\CO_S(C)^{[n]}$.
Taking the Kozsul resulution we get:
\[
\left[ \wedge^n (\CO(C)^{[n]})^{\vee} \to \wedge^{n-1} (\CO(C)^{[n]})^{\vee} \to \ldots \to (\CO(C)^{[n]})^{\vee} \to \CO_{S^{[n]}} \right] \cong \iota_{\ast} \CO_{C^{[n]}}
\]
Hence the dual $\CO_{C^{[n]}}^{\vee} := R \Hom( \iota_{\ast} \CO_{C^{[n]}}, \CO_{S^{[n]}} )$ reads
\[
\left( \iota_{\ast} \CO_{C^{[n]}} \right)^{\vee} \cong \left[ \CO_{S^{[n]}} \to \ldots \to \wedge^{n-1} \CO(C)^{[n]} \to \wedge^n \CO(C)^{[n]}  \right],
\]
where the term $\CO_{S^{[n]}}$ appears in degree $0$.

Comparing with \eqref{abc222} we conclude that:
\[
\Psi(\CO_C(C)^{\boxtimes n}, -1) = (\iota_{\ast} \CO_{C^{[n]}})^{\vee}[n].
\]
The final statement follows from Verdier duality:
\[
(\iota_{\ast} \CO_{C^{[n]}})^{\vee} \cong \iota_{\ast}( \omega_{\p^n} )[-n]. \qedhere
\]
\end{proof}

\subsection{Proof of Theorem~\ref{thm:main}}
Let $F = \ST_{\CO_S}$ and $\CL = \CO_S(C)$. By Proposition~\ref{prop:Key} we have
\begin{equation} \label{F eqn}
(F^{[n]})^H v(\CL_n \otimes \CO(-\delta) ) = v( \iota_{\ast}(\omega_{\p^n}) )
\end{equation}
where $\iota : \p^n = C^{[n]} \to S^{[n]}$ is the inclusion.
We now rewrite,
\begin{align*}
 (F^{[n]})^H = (F^H)^{[n]} & = 
B_{-\delta/2} \circ B_{\delta/2} \circ (F^H)^{[n]} \circ B_{-\delta/2} \circ B_{\delta/2} \\
& =
B_{-\delta/2} \circ \Theta_n(F^H) \circ B_{\delta/2}.
\end{align*}
Inserting into \eqref{F eqn}, we find that
\[ B_{-\delta/2} \circ \Theta_n(F^H) \circ B_{\delta/2} v(\CL_n \otimes \CO(-\delta)) = v( \iota_{\ast}(\omega_{\p^n}) ), \]
and hence
\[ \Theta_n(F^H) \circ v( \CL_n \otimes \CO(-\delta/2) ) = v( \iota_{\ast}(\omega_{\p^n} \otimes \iota^{\ast}\CO(-\delta/2) ) ). \]
We have that $F^H$ is an involution, so that:
\begin{equation} v( \CL_n \otimes \CO(-\delta/2) ) = \Theta_n(F^H) v( \iota_{\ast}(\omega_{\p^n} \otimes \iota^{\ast}\CO(-\delta/2) ) ). \label{jojo} \end{equation}

Note that
\[ v( \iota_{\ast}(\omega_{\p^n} \otimes \iota^{\ast}\CO(-\delta/2) ) ) = [ C^{[n]} ] + \ldots \]
where $(\ldots)$ stands for terms of higher codimension.
By Corollary~\ref{cor:Markman} 
we have that $\Theta_n(F^H)$ is degree-reserving, that is it sends $H^{k+2n}(S^{[n]})$ to $H^{-k+2n}(S^{[n]})$.
(Alternatively, 
this follows from Lemma~\ref{lemma:Htilde conjugate} and computing the commutator with $h$, see \cite[Lemma 3.4]{Markman2}).
Thus taking the degree $n$ component in \eqref{jojo} yields:
\[
\left[ v(\CL_n \otimes \CO(-\delta/2)) \right]_{\deg n} = \Theta_n(F^H) [ C^{[n]}].
\]
The following lemma hence implies equation \eqref{reduction} and hence Theorem~\ref{thm:main}.
\qed

\begin{lemma}
Let $C \subset S$ be a $(-2)$ curve and let $F = \ST_{\CO_S}$.
Then
\[ \Theta_n(F^H) [C^{[n]}] = [C^{[n]}]. \]
\end{lemma}
\begin{proof}
By a result of Grojnowski \cite{Groj} and Nakajima \cite{Groj} we have
\[ \sum_{n = 0}^{\infty} [C^{[n]}] = \exp\left( \sum_{m \geq 1} \frac{(-1)^{m-1}}{m} \Fq_{m}([C]) \right) \vacuum. \]
Hence $[C^{[n]}]$ is a linear combination of terms $\Fq_{k_1}([C]) \cdots \Fq_{k_{\ell}}([C]) \vacuum$.
Since $F^H$ acts by the identity of $H^2$, the claim follows now from Corollary~\ref{cor:Markman}.
\end{proof}

\section{Open questions}\label{sec:open questions}
\subsection{Multiple cover formula}
The independence of the formula \eqref{90sifsdf} from the divisibility of the class $\ell \in H_2(X,\BZ)$
suggests a connection to the multiple cover formula in Gromov-Witten theory of $K3^{[n]}$-type hyperk\"ahler varieties
conjectured in \cite{ObMC} and proven in \cite{QuasiK3} in many cases.

Assuming some familiarity with Gromov-Witten theory from \cite{ObMC} let $\Mbar_{0,2}(X,\beta)$
be the moduli space of $2$-marked genus $0$ stable maps $f: C \to X$ with $f_{\ast}[C] = \beta$.
The moduli space carries a $2n$-dimensional reduced virtual fundamental class $[\Mbar_{0,2}(X,\beta)]^{\vir}$.
Its pushforward along the product of the evaluation maps $\ev_i : \Mbar_{0,2}(X,\beta)$, $i=1,2$ defines the Lagrangian cycle
\[ Z_{\beta} := (\ev_1 \times \ev_2)_{\ast} [\Mbar_{0,2}(X,\beta)]^{\vir} \in A_{2n}(X \times X). \]
If the moduli space $\Mbar_{0,2}(X,\beta)$ is of expected dimension $2n$,
then the virtual class is just the ordinary fundamental class.
For example, Theorem~\ref{thm:main} shows that for $\beta = \ell$ primitive, we have in cohomology
\[
Z_{\ell} = \mathrm{pr}_{1}^{\ast}\left( [ v(\CO(L)) ]_n \right) \cup \mathrm{pr}_{2}^{\ast}\left( [ v(\CO(L)) ]_n \right),
\]
which shows that $Z_{\ell}$ encodes information about the class $[P]$.
A basic idea is to explore whether information about $Z_{\ell}$ for $\ell$ imprimitive
can be used to prove Theorem~\ref{thm:main} in general.
In particular, we have the following formula.

\begin{thm}[\cite{QuasiK3}] \label{thm:MC} Let $\beta \in H_2(X,\BZ)$ be an effective curve class.
For any $k|\beta$ let $X_k$ be a hyperk\"ahler variety of $K3^{[n]}$-type
and let $\varphi_k : H^2(X,\BR) \to H^2(X_k,\BR)$ be a real isometry such that
\begin{itemize}
\item $\varphi_k(\beta/k)$ is a primitive effective curve class,
\item $\pm [ \beta/k ] = \pm [ \varphi_k(\beta/k) ]$ in $H_2(X,\BZ)/H^2(X,\BZ) \cong H_2(X_k,\BZ)/H^2(X_k,\BZ) \cong \BZ/(2n-2)\BZ$.
\end{itemize}
Extend $\varphi_k$ as a parallel transport lift $H^{\ast}(X,\BR) \to H^{\ast}(X_k, \BR)$, see \cite{ObMC}.
Then
\[ Z_{\beta} = \sum_{k|\beta} \frac{1}{k} (\varphi_k \boxtimes \varphi_k)^{-1}( Z_{\varphi_k(\beta/k)} ). \]
\end{thm}

Theorem~\ref{thm:MC} shows that for $\ell$ with arbitrary divisibility we have
\[
Z_{\ell} = \mathrm{pr}_{1}^{\ast}\left( [ v(\CO(L)) ]_n \right) \cup \mathrm{pr}_{2}^{\ast}\left( [ v(\CO(L)) ]_n \right) + (... )
\]
where the first term corresponds to the $k=1$ term in the Theorem, and (...) stands for the $k>1$ which should be
the multiple cover contributions coming from the rational curves in class $\beta/k$.
A pathway to Theorem~\ref{thm:main} hence lies in understanding the rational curves in class $\beta/k$ better.

\subsection{Arbitrary hyperk\"ahler varieties}
Let $X$ be any hyperk\"ahler variety containing a Lagrangian plane $P \subset X$.
Let $\overline{ ( - ) } : H^{\ast}(X,\BQ) \to SH^{\ast}(X,\BQ)$ be the orthogonal projection to the
subspace generated by divisor classes.
With the same notation of the introduction, Song proved in \cite[Thm. 2.2.4]{SongPhD} that
\[ \overline{ [ P ]} = 
\overline{ \left[ \frac{\mu^n}{c_X} \exp(L/\mu) \sqrt{\td_X} \right]_n } \]
where $c_X = \frac{n! 2^n}{(2n)!} C(1)$ and $\mu = -(\ell,\ell)/2r_X$.
Moreover, $\mu=1$ conjecturally by \cite{SongPhD}.

However, in general we can not expect a formula for $[P] \in H^{\ast}(X,\BQ)$ only in terms of the class $L$ dual to the class of the line $\ell$.
For example, as discussed in \cite[Example 2.2.8]{SongPhD} based on \cite{HT3} for generalized Kummer $4$-folds we have
$[P] = 1/6 L^2 + 1/72 c_2(X) + z$ where the class $z$ is non-zero and not monodromy invariant.
The issue seems to arise since there can be several Lagrangian planes $P_1, \ldots, P_N \subset X$
which have the same class of line $\ell$.
A speculation is hence whether the average over $[P_i]$ satisfies Song's formula.

\begin{question} Let $X$ be any hyperk\"ahler variety. For $\ell \in H_2(X,\BZ)$, let $P_1, \ldots, P_N \subset X$ be the Lagrangian planes
such that a line on them has class $\ell$. Assume $\mu=1$. Does
\[ \frac{1}{N} \sum_i [P_i] = \left[ \frac{1}{c_X} \exp(L) \sqrt{\td_X} \right]_n \]
hold?
\end{question}

\subsection{Chow ring}
As already asked by Song \cite{SongPhD}, we can ask for Conjecture~\ref{conj:Song} also
as an equality of Chow rings $A^{\ast}(X)$.
To use the methods of this paper for the case $X=S^{[n]}$ and $P = C^{[n]}$ for a $(-2)$-curve $C \subset S$,
we would need to extend the result of Talman \cite{Talman} to the Chow ring.
More precisely, Talman shows that the induced action on cohomology of any auto-equivalence of a hyperk\"ahler variety
intertwines the action of the LLV algebra. To lift these results to Chow hence requires two statements:
\begin{question} Let $X$ be a hyperk\"ahler variety.
\begin{enumerate}
\item[(i)] Does the LLV algebra action on cohomology lift naturally to an action on the Chow ring $A^{\ast}(X)$?
\item[(ii)] Assuming (i), let $F : D^b(X) \to D^b(X)$ be any auto-equivalence. Does the induced action $F_{\ast} : A^{\ast}(X) \to A^{\ast}(X)$
intertwine the action of the LLV algebra?
\end{enumerate}
\end{question}
By natural action we mean one where the lift of the Lefschetz grading operator $h$ should
give the expected Beauville-Voisin decomposition \cite{Beauville, Voisin} of the Chow ring $A^{\ast}(X)$.
The first question was expected since the work \cite{OLLV}, where an LLV algebra action on Chow was constructed for the Hilbert scheme $S^{[n]}$.
This action turns out to have the expected properties \cite{NOY}.
Hence for $S^{[n]}$ the second question can be stated unconditionally.
For partial evidence for (i) in the case of the Fano variety of lines of a cubic fourfold, see also \cite{Kretschmer}.

\appendix
\section{Orbifold Mukai vectors} \label{sec:orbifold mukai}
Let $G$ be a finite group acting on a smooth projective variety $X$,
and let $\Coh_G(X)$ be the category of $G$-equivariant  coherent sheaves on $X$.
In this appendix we define a Mukai vector
\[ v : K(\Coh_G(X)) \to H_G^{\ast}(X) \]
taking values in the orbifold cohomology of $X$.
We show that it is well-behaved with respect to Fourier-Mukai transforms in the usual sense.
We discuss the case of the Bridgeland-King-Reid isomorphism specifically.
Our discussion can be viewed as a reformulation of work of Baum-Fulton-Quartz on the equivariant Riemann-Roch theorem \cite{BFQ}.
A similar discussion can also be found in \cite{Popa}, and no originality is claimed here.
We also refer to \cite{FG, CR, JKK} for further discussions on orbifold cohomology.\footnote{However, contrary to these references we will always work here with the much simpler classical (non-stringy) product.}

\subsection{Orbifold cohomology}
The orbifold cohomology of the pair $(X,G)$ is defined by
\[ H_G^{\ast}(X) = \left( \oplus_{g \in G} H^{\ast}(X^g) \right)^{G} \]
where an element $h \in G$ acts on an element $\alpha \in H^{\ast}(X^g)$ 
by $h_{\ast}(\alpha)$ and $h : X^g \to X^{hgh^{-1}}$ is the isomorphism defined by the group action.
The multiplication on $H_G^{\ast}(X)$ is defined factor-wise:
\[ (\alpha_g )_{g \in G} \cdot ( \beta_g )_{g \in G} = ( \alpha \cdot \beta )_{g \in G}. \]

\begin{rmk}
A second (so called "stringy") product was defined in \cite{JKK}.
It corresponds to the cup product on the Hilbert scheme,
that is after the Bridgeland-King-Reid isomorphism.
\end{rmk}

\subsection{Trivial action}
For the trivial action of a finite group $G$ on a variety $X$ there is an isomorphism
\[ \Coh_G(X) \cong \bigoplus_{i} \Coh(X) \otimes V_i \]
where $V_i$ are the irreducible representations of $G$.
Hence we have an isomorphism:
\[ K(\Coh_G(X)) \cong K(\Coh(X)) \otimes R[G] \]
where $R[G]$ is the representation ring. For any given element $g \in G$ we have the trace morphism
$\tr_g : R[G] \to \BC$.
We obtain the trace map:
\[ \Tr_g : K(\Coh_{G}(X)) \cong K(\Coh(X)) \otimes R[G] \xrightarrow{\id \times \tr_g} K(\Coh(X)) \]
Since the trace only depend on the action of $g$ on $V_i$,
this definition is independent of the group containing $G$, e.g. we can take $G = \langle g \rangle$.
A direct check shows that $\Tr_g$ is a ring homomorphism:
$\Tr_g( A \times B ) = \Tr_g(A) \otimes \Tr_g(B)$.

\begin{example}
For a cyclic group $G = \BZ / n \BZ$ generated by $g$
we have one irreducible representation $V_{\chi}$ for each root of unity $\chi = e^{2 
pi i k/n}$, where $g$ acts by $\chi$.
Hence, if $W$ is a $G$-vector bundle which decomposes into eigenspaces $W_{\chi}$ under $g$, then
$\Tr_g(W) = \sum_{\chi} \chi W_{\chi}$.
\end{example}

\subsection{Chern character}
The (classical) orbifold Chern character
\[ \ch^{G} : K(\Coh_G(X)) \to H_G^{\ast}(X) \label{equivariant Chern character} \]
is the ring homomorphism defined by:
\[ \ch^G(W) = \big( \ch_g(W) )_{g \in G} := \Big( \ch(\Tr_{g}( W|_{X^g} )) \Big)_{g \in G}. \]

Define also the Baum-Fulton-Quartz Riemann-Roch morphism
\[ \tau^{G} : K(\Coh_G(X)) \to H^G(X) \]
by assigning to each $G$-equivariant vector bundle $W$:
\[ \tau^G(W) = \Big( \tau_{g}(W) \Big) := \left( \frac{ \ch( \Tr_g(W|_{X^g}) ) }{ \ch( \Tr_g( \sum_i (-1)^i \Lambda^i N^{\ast}_{X^g/X}) )} \td(X^g) \right)_{g \in G} \]
This is the orbifold analogue of the Riemann-Rich morphism $\tau = \ch(-) \td_X$,

We have the compatibility:
\[ \tau^G( V \otimes W) = \ch^G(V) \cdot \tau^G(W). \]

\subsection{Pullback and pushforward}
Let $f : X \to Y$ be a $G$-equivariant morphism, and consider the induced morphisms:
$f_{X^g} : X^g \to Y^g$.
Define the pullback factorwise by
\[ f^{\ast} : H_G^{\ast}(Y) \to H_G^{\ast}(X),  \quad f^{\ast}(\beta_g) = ( f_{X^g}^{\ast} \beta_g )_{g \in G}. \]
By a straightforward check one has:
\begin{lemma} $\ch^G(f^{\ast} W) = f^{\ast} \ch^G(W)$.
\end{lemma}

Similarly, define the pushfoward factorwise:
\[ f_{\ast} : H_G^{\ast}(X) \to H_G^{\ast}(Y), \quad f_{\ast}(\alpha_g) = ( f_{X^g, \ast} \alpha_g )_{g \in G}. \]

\begin{prop}[\cite{BFQ}]
For any proper $G$-equivariant morphism $f : X \to Y$ and $G$-equivariant sheaf $W$ we have
\[ \tau^G(R f_{\ast} W) = f_{\ast} \tau^G(W). \]
\end{prop}

\subsection{Mukai vector}
Given an $G$-equivariant vector bundle $W$ we define the Mukai vector by
\[ v(W) := \Big( v_g(W) \Big)_{g \in G} := \left( 
\ch( \Tr_g(W|_{X^g}) ) \cdot \sqrt{ \frac{\td(X^g)}{ \ch( \Tr_g( \sum_i (-1)^i \Lambda^i N^{\ast}_{X^g/X}) )} } \right)_{g \in G}. \]

Consider $G$-actions on varieties $X,Y$ and endow $X \times Y$ with the diagonal action.
In particular, $(X \times Y)^g = X^g \times Y^g$ and
\[ N_{(X \times Y)^g / (X \times Y) } \cong \mathrm{pr}_1^{\ast}( N_{X^g/X}) \oplus 
\mathrm{pr}_2^{\ast}( N_{Y^g/Y} ). \]
Given an object $\CE \in D^b_G(X \times Y)$ consider the associated Fourier-Mukai transform:
\[ \FM_{\CE} : D^b_G(X) \to D^b_{G}(Y), \quad A \mapsto \pr_{2 \ast}( \pr_1^{\ast}(A) \otimes \CE). \]

Given a class $\alpha \in H_G^{\ast}(X \times Y)$ let it act as a correspondence by:
\[ \FM_{\alpha} : H_G^{\ast}(X) \to H_G^{\ast}(Y), \quad \beta \mapsto \pr_{2 \ast}( \pr_1^{\ast}(\beta) \cdot \alpha ). \]

\begin{prop} We have a commutative diagram:
\[
\begin{tikzcd}
D^b_G(X) \ar{d}{v(-)} \ar{r}{\FM_{\CE}} & D^b_G(Y) \ar{d}{v ( - )} \\
H_G^{\ast}(X) \ar{r}{\FM_{v(\CE)}} & H_G^{\ast}(Y).
\end{tikzcd}
\]
\end{prop}
\begin{proof}
Note that we have
\[ \tau_g(\CE) = v(\CE) \cdot
\sqrt{ \frac{\td(X^g)}{ \ch( \Tr_g( \sum_i (-1)^i \Lambda^i N^{\ast}_{X^g/X}) )}}
\cdot 
\sqrt{ \frac{\td(Y^g)}{ \ch( \Tr_g( \sum_i (-1)^i \Lambda^i N^{\ast}_{Y^g/Y}) )} }.
\]
Hence we find
\begin{align*}
v_g( \pr_{2 \ast}( \pr_1^{\ast}(A) \otimes \CE) )
& =
 \sqrt{ \frac{\td(Y^g)}{ \ch( \Tr_g( \sum_i (-1)^i \Lambda^i N^{\ast}_{Y^g/Y}) )} }^{-1} 
 \tau_{g}( \pr_{2 \ast}( \pr_1^{\ast}(A) \otimes \CE) ) \\
& =
\sqrt{ \frac{\td(Y^g)}{ \ch( \Tr_g( \sum_i (-1)^i \Lambda^i N^{\ast}_{Y^g/Y}) )} }^{-1} 
\pr_{2 \ast}( \tau_{g}(\pr_1^{\ast}(A) \otimes \CE) ) \\
& =
\sqrt{ \frac{\td(Y^g)}{ \ch( \Tr_g( \sum_i (-1)^i \Lambda^i N^{\ast}_{Y^g/Y}) )} }^{-1} 
\pr_{2 \ast}( \tau_{g}(\CE) \pr_1^{\ast}( \ch^G(A)) ) \\
& = \Phi_{v(\CE)}( v(A) ).
\end{align*}
\end{proof}

\subsection{Permutation action}
Let $X$ be a smooth projective variety, and consider the action of $G=S_n$ on $Y:=X^n$ by permutation of factors.
Given a sheaf $F$ on $X$ the product
\[ F^{\boxtimes n} = \otimes_i \pr_i^{\ast}(F) \]
has a canonical $S_n$-linearization, giving rise to $(F^{\boxtimes n},1) \in D_{G}(Y)$, see \eqref{linearization}.

The Mukai vector $v(F^{\boxtimes n})$ 
can be described as follows.
First one has the following 'localization' formula that describes the component of maximal cycle type.

\begin{prop} \label{prop:g comp} Let $g \in G$ have cycle type $(n)$.
\begin{gather*}
\ch_g(F^{\boxtimes n},1) = \psi^n( \ch(F)) \\
\tau_{g}(F^{\boxtimes n},1) 
= \frac{1}{n^{\dim X}} \psi^n( \tau(F) ) \\ 
v_g( F^{\boxtimes n},1)
= \frac{1}{n^{\dim X/2}} \psi^n( v(F)).
\end{gather*} 
\end{prop}
Here $\psi^n$ is the Adams operation which acts by multiplication by $n^i$ on $H^{2i}$.

\begin{proof}
The first goes back to a PhD thesis of Moonen,
see \cite[Lemma 5.11]{MS} and the references therein.
However, it can also be derived directly quite easily,
see the discussion in \cite{Nori}.
To sketch the details,
for a line bundle $\CL$ 
the $g$-action on the restriction $\CL^{\boxtimes n}|_{\Delta_{1 ... n}}$ is trivial,
so $\ch_g(\CL^{\boxtimes n},1) = \ch( \CL^{\otimes n}) = \psi^n( \ch(\CL))$.
The general case follows since classes of (topological) line bundles generate the topological $K$-theory.

For the second claim one has that (e.g. \cite[Lemma 3.3]{Nori})
\[ \ch \Tr_g( \sum_i (-1)^i \Lambda^i N_{\Delta/X^n} ) = \theta^n(X), \]
where $\theta^n(X)$ are Bott's cannibalistic classes, and by
\cite[Lemma 1.4]{Nori} these satisfy
\[ \psi^n( \td(X)) = n^{\dim X} \frac{ \td(X)}{\theta^n(X)}. \]
This concludes the claim since $\psi^n$ is a ring homomorphism.

For the last part, observe that
\[
v_g( F^{\boxtimes n})=
\ch_g( F^{\boxtimes n}) \sqrt{ \frac{1}{n^{\dim X}} \psi^n( \td(X)) }
\]
and that $\sqrt{ - }$ and $\psi^n$ commute.
\end{proof}

In the general case we have the following 'multiplicativity' \cite[Example 3.1]{MS} .
Let $g \in S_n$ be of cycle type $(k_1, \ldots, k_{\ell})$.
Then under the isomorphism $Y^g \cong X^{\ell}$ 
we have
\begin{equation} \ch_g(F^{\boxtimes n},1) = \psi^{k_1}( \ch(F)) \boxtimes \cdots \boxtimes \psi^{k_{\ell}}( \ch(F)). \label{multiplicativity} \end{equation}

\subsection{Symmetric actions on $X^n$} \label{subsec:symmetric action on Sn}
Let $\CE \in D^b( X \times X)$ be the kernel of a Fourier-Mukai transform $\FM_{\CE} : D^b(X) \to D^b(X)$,
and as in \eqref{induced box transform} consider the induced $G:=S_n$-equivariant functor on $Y=X^g$:
\[ \FM_{\CE}^{\boxtimes n} : D_{G}(Y) \to D_G(Y), \quad A \mapsto \pr_{2 \ast}( \pr_1^{\ast}(A) \otimes (\CE^{\boxtimes n},1) ). \]

Let $\phi = \FM_{v(\CE)} : H^{\ast}(X) \to H^{\ast}(X)$ the action on cohomology,
and consider the induced action on orbifold cohomology
\[ \phi^{\boxtimes n} := \FM_{v(\CE^{\boxtimes n},1)} : H_G^{\ast}(Y) \to H_G^{\ast}(Y). \]

We have the following description. Define
\[ \psi^n(\phi) = n^{\deg_{\BR}/2} \circ \phi \circ n^{-\deg_{\BR}/2} \]
where $n^{\deg_{\BR}}$ acts on $H^i(X)$ by multiplication by $n^i$.

\begin{lemma} We have
\begin{equation} \phi^{\boxtimes n}( \alpha_g )_{g \in G} = ( \widetilde{\phi}(\alpha_g) )_{g \in G} \label{induced action} \end{equation} 
where if $g$ is of cycle type $(k_1, \ldots, k_{\ell})$,
then $\widetilde{\phi} = \psi^{k_1}(\phi) \boxtimes \ldots \boxtimes \psi^{k_{\ell}}(\phi)$
under the isomorphism $Y^g \cong X^{\ell}$.
\end{lemma}
\begin{proof}
Since pullback and pushforward acts factorwise, also $\FM_{v(\CE^{\boxtimes n},1)}$ acts factorwise.
By \eqref{multiplicativity} it suffices moreover to  consider $g$ of maximal cycle type $(n)$
for which we have by Proposition~\ref{prop:g comp} that
\[ v_g(\CE^{\boxtimes n}) = \frac{1}{n^{\dim(X)}} \psi^n( v(\CE)). \]
By a direct check this acts on $H^{\ast}(X)$ as a correspondence by $n^{\deg_{\BR}/2} \circ \phi \circ n^{-\deg_{\BR}/2}$.
\end{proof}

\begin{rmk}
If $\phi \in \End H^{\ast}(X)$ we call $\phi^{\boxtimes n} \in \End H_G(X^n)$ defined by \eqref{induced action}
the induced action on $H_G(X^n)$.
\end{rmk}

%
%

\subsection{Taking invariants}
For the trivial $G$-action on a space $X$ we can consider the functor that takes invariants:
\[ (-)^G : \Coh_G(X) \to \Coh(X), \quad W \mapsto W^G. \]

\begin{lemma} For any $W$-equivariant sheaf,
	\begin{gather*}
		\ch( W^G ) = \frac{1}{|G|} \sum_{g \in G} \ch_g(W), \quad \quad
		\tau( W^G ) = \frac{1}{|G|} \sum_{g \in G} \tau_g(W) \\
		v( W^G ) = \frac{1}{|G|} \sum_{g \in G} v_g(W).
	\end{gather*}
\end{lemma}
\begin{proof}
	We use that for any $G$-representation we have
	\[ \dim(V^G) = \frac{1}{|G|} \sum_{g \in G} \tr(g| V). \]
	Write $W = \bigoplus W_{\chi} \otimes V_\chi$ where $V_{\chi}$ are the irreducible of $G$.
	We get
	\begin{align*}
		\ch( W^G ) = \ch( W_1 ) & = \frac{1}{|G|} \sum_{g \in G} \ch(W_{\chi}) \tr(g| V_{\chi}) \\
		& = \frac{1}{|G|} \sum_{g \in G} \ch_g(W).
	\end{align*}
	Similarly, note that for the trivial $G$-action we have $X^g=X$,
	so $\tau_g(W) = \ch_g(W) \cdot \td(T_X)$, hence the second claim follows by multiplication by $\td(T_X)$.
\end{proof}

Define the sum map
\[ \sigma : H_G^{\ast}(X) \to H^{\ast}(X), (\alpha_g) \mapsto \frac{1}{|G|} \sum_{g \in G} \alpha_g. \]
Then the above says that we have a commutative diagram
\[
\begin{tikzcd}
K(\Coh_G(X)) \ar{r}{( - )^G} \ar{d}{v( - )} & K(\Coh(X)) \ar{d}{v} \\
H_G^{\ast}(X) \ar{r}{\sigma} & H^{\ast}(X).
\end{tikzcd}
\]

\subsection{The Bridgeland-King-Reid isomorphism} \label{sec:appendix BKR}
For a smooth projective surface $S$ recall from \eqref{section:BKR} the Bridgeland-King-Reid equivalence
\[ \Psi : D_G(S^n) \to D^b(S^{[n]}), \quad \CF \mapsto \pr_{2\ast}( \pr_1^{\ast}(\CF) \otimes \CO_{Z_n} )^G. \]
By the discussion above we obtain the induced cohomological transform:
\[ \Psi^H = \sigma \circ \FM_{v(\CO_Z)} : H_G^{\ast}(S^n) \to H_G^{\ast}(S^{[n]}) \to H^{\ast}(S^{[n]}), \]
such that the following diagram commutes:
\[
\begin{tikzcd}
K(\Coh_G(S^n)) \ar{d}{v} \ar{r}{ \Psi } & K( \Coh(S^{[n]})) \ar{d}{v} \\
H_G^{\ast}(S^n) \ar{r}{\Psi^H} & H^{\ast}(S^{[n]}).
\end{tikzcd}
\]

The last thing we need for Section~\ref{subsec:BKR in cohomology} is the following:
\begin{lemma}
$\Psi^H$ is an isomorphism
\end{lemma}
\begin{proof}
The most direct way to see this is by observing that
the equivariant Chern character \eqref{equivariant Chern character}
defines an isomorphism between the $G$-equivariant topological $K$-theory and the orbifold cohomology \cite{AS}.
Hence the claim follows from the fact that the Bridgeland-King-Reid isomorphism
induces an isomorphism in topological $K$-theory, see \cite[Sec.10]{BKR}.

Alternatively, recall the universal family $Z_n \subset S^{[n]} \times S^n$ 
and let $p,q$ be the projection of $S^{[n]} \times S^n$ to the factors.
We then compute directly:
\begin{align*}
\tau( R p_{\ast}( q^{\ast}(\CF))^G 
& = \frac{1}{|G|} \sum_g \tau_g( R p_{\ast}( \CO_{Z_n} \otimes q^{\ast}(\CF)) ) \\
& =
\frac{1}{|G|} \sum_g p_{\ast}( \tau_g(\CO_{Z_n^g}) q^{\ast}( \ch_g(\CF))  \\
& =
\frac{1}{|G|} \sum_g p_{\ast}( \tau(\CO_{Z_n^g}) q^{\ast}\left( 
\frac{\ch_g(\CF)}{ \ch( \Tr_g( \sum_i (-1)^i N_{(S^n)^g/S^n}))}
\right)
\end{align*}
where we used the construction of $\tau$ as in \cite{BFQ} in the last step.
We have
\[ \tau(\CO_{Z_n^g}) = [Z_n^g] + ... \]
where $...$ stands for classes of higher codimension.
Hence $\Psi^H$ is upper triangular, where on the diagonal we have the usual cohomological isomorphism between $H^g(S^n)$ and $H^{\ast}(S^{[n]})$,
as discussed for example in \cite{FG} or \cite{dCM}.
\end{proof}

Mathemathisches Institut, Universit\"at Bonn

georgo@math.uni-bonn.de \\


\begin{thebibliography}{99}

\bibitem{AS} 
M.F. Atiyah and G. Segal, {\em On equivariant Euler characteristics}, J. Geom. Phys. 6 (1989) 671–677.


\bibitem{Bakker}
B. Bakker,
{\em A classification of Lagrangian planes in holomorphic symplectic varieties},
J. Inst. Math. Jussieu 16 (2017), no. 4, 859--877. 

\bibitem{BJ}
B. Bakker, A.Jorza, {\em Lagrangian 4-planes in holomorphic symplectic varieties of K3[4]-type}, Cent. Eur. J. Math. 12 (2014), no. 7, 952–975.

\bibitem{Beauville}
A. Beauville, {\em On the splitting of the Bloch–Beilinson filtration, Algebraic cycles and motives},
Vol. 2, 38–53, London Math. Soc. Lecture Note Ser., 344, Cambridge Univ. Press, Cambridge, 2007.

\bibitem{Beckmann1}
T. Beckmann,
{\em Derived categories of hyper-Kähler manifolds: extended Mukai vector and integral structure},
arXiv:2103.13382

\bibitem{Beckmann2}
T. Beckmann, {\em  Atomic Objects on hyper-Kähler manifolds}, arXiv:2202.01184


\bibitem{BFQ}
P. Baum, W. Fulton, G. Quart,
{\em Lefschetz-Riemann-Roch for singular varieties},
Acta Math. 143 (1979), no. 3-4, 193--211. 


\bibitem{BKR}
T. Bridgeland, A. King, M. Reid,
{\em The McKay correspondence as an equivalence of derived categories},
J. Amer. Math. Soc. 14 (2001), no. 3, 535--554.


\bibitem{CR} W. Chen, Y. Ruan, {\em A new cohomology theory of orbifold}, Comm. Math. Phys. 248 (2004), no. 1, 1--31.

\bibitem{dCM}
M. A. de Cataldo, L. Migliorini,
{\em The Chow groups and the motive of the Hilbert scheme of points on a surface},
J. Algebra 251 (2002), no. 2, 824--848. 

\bibitem{FG} B. Fantechi, L. G\"ottsche,
{\em Orbifold cohomology for global quotients},
Duke Math. J. 117 (2003), no. 2, 197--227. 


\bibitem{Groj} Grojnowski, I. {\em Instantons and affine algebras. I. The Hilbert scheme and vertex operators}, Math. Res. Lett. 3 (1996), no. 2, 275--291.

\bibitem{Haiman}
M. Haiman, {\em Hilbert schemes, polygraphs and the Macdonald positivity conjecture}, J. Amer. Math. Soc. 14 (2001), no. 4, 941--1006.

\bibitem{HHT}
D. Harvey, B. Hassett, Y. Tschinkel,
{\em Characterizing projective spaces on deformations of Hilbert schemes of K3 surfaces},
Comm. Pure Appl. Math. 65 (2012), no. 2, 264--286.

\bibitem{HT}
B. Hassett, Y. Tschinkel,
{\em Moving and ample cones of holomorphic symplectic fourfolds},
Geom. Funct. Anal. 19 (2009), no. 4, 1065--1080. 

\bibitem{HT2}
B. Hassett, Y. Tschinkel, {\em Intersection numbers of extremal rays on holomorphic symplectic varieties}, Asian J. Math. 14 (2010), no. 3, 303--322.

\bibitem{HT3}
B. Hassett, Y. Tschinkel,
{\em Hodge theory and Lagrangian planes on generalized Kummer fourfolds},
Mosc. Math. J. 13 (2013), no. 1, 33--56, 189. 

\bibitem{JKK}
T. Jarvis, R. Kaufmann, T. Kimura,
{\em Stringy K-theory and the Chern character},
Invent. Math. 168 (2007), no. 1, 23--81. 

\bibitem{Kretschmer}
A. Kretschmer,
{\em The Chow ring of hyperkähler varieties of K3[2]-type via Lefschetz actions},
Math. Z. {\bf 300} (2022), no. 2, 2069--2090. 


\bibitem{Krug}
A. Krug,
{\em Remarks on the derived McKay correspondence for Hilbert schemes of points and tautological bundles},
Math. Ann. 371 (2018), no. 1-2, 461--486. 


\bibitem{Lehn}
M. Lehn,
{\em Chern classes of tautological sheaves on Hilbert schemes of points on surfaces},
Invent. Math. 136 (1999), no. 1, 157–207. 


\bibitem{LL}
E.~Looijenga, V.~A.~Lunts,
{\em A Lie algebra attached to a projective variety},
Invent. Math. 129 (1997), no. 2, 361--412.


%


\bibitem{Markman1}
E. Markman 
{\em Stable vector bundles on a hyper-Kahler manifold with a rank 1 obstruction map are modular},
arXiv:2107.13991

\bibitem{Markman2} 
E. Markman,
{\em  Rational Hodge isometries of hyper-Kahler varieties of K3[n]-type are algebraic},
arXiv:2204.00516

\bibitem{OLLV}
G. Oberdieck, {\em A Lie algebra action on the Chow ring of the Hilbert scheme of points of a K3 surface}, Comment. Math. Helv. 96 (2021), no. 1, 65--77.




\bibitem{ObMC} G. Oberdieck, {\em Gromov-Witten theory and Noether-Lefschetz theory for holomorphic-symplectic varieties},
Forum Math. Sigma 10 (2022), Paper No. e21.


\bibitem{QuasiK3}
G. Oberdieck, {\em Multiple cover formulas for K3 geometries, wallcrossing, and Quot schemes}, Preprint, arXiv:2111.11239


\bibitem{Ploog}
D. Ploog,
{\em Equivariant autoequivalences for finite group actions},
Adv. Math. 216 (2007), no. 1, 62--74. 

\bibitem{PS} D. Ploog, P. Sosna, {\em On autoequivalences of some Calabi-Yau and hyperkähler varieties}, Int. Math. Res. Not. IMRN 2014, no. 22, 6094--6110.

\bibitem{MS}
L. Maxim, J. Sch\"urmann,
{\em Equivariant characteristic classes of external and symmetric products of varieties},
Geom. Topol. 22 (2018), no. 1, 471--515.

\bibitem{Nori}
M. Nori,
{\em The Hirzebruch-Riemann-Roch theorem},
Dedicated to William Fulton on the occasion of his 60th birthday.
Michigan Math. J. 48 (2000), 473--482.


\bibitem{Nak}
H.~Nakajima, \emph{ Heisenberg algebra and {H}ilbert schemes of points on
  projective surfaces},
\newblock Ann. of Math. (2) 145 (1997), no. 2, 379--388.


\bibitem{NOY}
A. Negut, G. Oberdieck, Q. Yin,
{\em Motivic decompositions for the Hilbert scheme of points of a K3 surface},
J. Reine Angew. Math. 778 (2021), 65--95. 


\bibitem{Popa}
M. Popa,
{\em Derived equivalences of smooth stacks and orbifold Hodge numbers} in Higher dimensional algebraic geometry—in honour of Professor Yujiro Kawamata's sixtieth birthday, 357--380,
Adv. Stud. Pure Math., 74, Math. Soc. Japan, Tokyo, 2017.

\bibitem{SongPhD}
J. Song, {\em Geometry of hyperkähler manifolds}, PhD thesis, Universit\'e Paris Cit\'e,
available from \url{https://webusers.imj-prg.fr/~jieao.song/docs/thesis_Jieao_Song.pdf}

\bibitem{Talman}
L. Taelman,
{\em Derived equivalences of hyperk\"ahler varieties},
arXiv:1906.08081


\bibitem{V}
M.~Verbitsky, {\em Cohomology of compact hyper-K\"ahler manifolds and its applications},
Geom. Funct. Anal. 6 (1996), no. 4, 601--611.

\bibitem{Voisin}
C. Voisin, {\em On the Chow ring of certain algebraic hyper-Kähler manifolds}, Pure Appl. Math. Q. {\bf 4} (2008), no. 3, Special Issue: In honor of Fedor Bogomolov. Part 2, 613--649.


\end{thebibliography}
\end{document}